\documentclass[12pt]{article} 

\usepackage{times}
\usepackage{amssymb}
\usepackage{amsmath}
\usepackage{graphicx}
\usepackage{epsfig}
\usepackage{subfigure}

\def\shadowbox{\hbox{\rule[-0.0ex]{0.1ex}{1.2ex}%
\hspace{-0.1ex}\rule[-0.0ex]{1.2ex}{0.1ex}%
\hspace{0.0ex}\rule[-0.0ex]{0.1ex}{1.2ex}\hspace{-1.3ex}%
\rule[1.15ex]{1.25ex}{0.1ex}\hspace{-0.0ex}\rule[-0.25ex]{0.3ex}{1.1ex}%
\hspace{-1.2ex}\rule[-0.25ex]{1.1ex}{0.25ex}}}
\def\qed{\ifmmode \hbox{\hfill\shadowbox}
     \else \hphantom{x}\hfill\shadowbox \fi}

\newtheorem{theorem}{Theorem}[section]
\newtheorem{lemma}[theorem]{Lemma}
\newtheorem{definition}[theorem]{Definition}
\newtheorem{proposition}[theorem]{Proposition}
\newtheorem{corollary}[theorem]{Corollary}
\newtheorem{note}[theorem]{Note}

\allowdisplaybreaks[4]

\def\R {\mathbb{R}}

\def\C {\mathbb{C}}

\def\P {\mathbb{P}}

\def\Fn {F_{n}}
\def\Fnstar {F_{n}^{*}}
\def\Vn {V_{n}}

\def\Ub {\mathcal{U}}

\def\Ubn {\mathcal{U}_{n}}
\def\Ubnstar {\mathcal{U}_{n}^{*}}
\def\Fstar {F^{*}}
\def\Vstar {V^{*}}

\def\Pn {P_{n}}
\def\Qn {Q_{n}}
\def\On {\Omega_{n}}
\def\Tn {T_{n}}
\def\Un {U_{n}}
\def\Ubn {\mathcal{U}_n}
\def\Wn {W_{n}}
\def\Ws {W^{*}}
\def\wi {w_{i}}
\def\wis {w_{i}^{*}}

\def\wiss {w_{i^*}^{*}}
\def\wj {w_{j}}
\def\wjs {w_{j}^{*}}
\def\Aii {A_{i}^{-1}}

\def\tr {\textnormal{tr}}
\def\E {\mathbb{E}}
\def\a {(2(p+q)-4pq)}
\def\b {(p-q)^2}
\def\EXP {\frac{\sqrt{(1-\frac{r_{-}}{x})(\frac{r_{+}}{x} -1 )}}{2\pi(1-x)}}

\def\FOT {F_{\Omega T}}
\def\FOnTn {F_{\Omega_n T_n}}
\def\UbOT {\mathcal{U}_{\Omega T}}
\def\UOT {U_{\Omega T}}
\def\UOnTn {U_{\Omega_n T_n}}
\def\UbOnTn {\mathcal{U}_{\Omega_n T_n}}

\def\etapq {\eta_{p,q}}

\def\Qt {\tilde{Q}}
\def\Qb {\bar{Q}}
\def\Pb {\bar{P}}
\def\Qtistar {\tilde{Q}_{i^*,i^*}}
\def\Pt {\tilde{P}}

\def\supp {\textnormal{supp}}

\def\iistar {(i^*,i^*)}

\begin{document}

\title{\bf Limiting Empirical Singular Value Distribution of Restrictions of Unitary Matrices  }



\author{Brendan Farrell
\thanks{   Heinrich-Hertz-Lehrstuhl f\"ur Informationstheorie und Theoretische Informationstechnik,
Technische Universit\"at Berlin,
Einsteinufer 25, 10587 Berlin, Germany.
\newline email: farrell@mk.tu-berlin.de}}




\date{\today}

\maketitle


\begin{abstract}
We determine the limiting empirical singular value distribution for 
random unitary matrices with Haar distribution and discrete Fourier transform (DFT) 
matrices when a random set of columns and rows is removed.
\end{abstract}







\begin{section}{Introduction}

An $n\times n$ Hermitian matrix $A$ determines a distribution on the real line by 
\begin{equation*}
f_A(x)=\frac{1}{n}\sum_{i=1}^n \delta_{\lambda_i}(x),
\end{equation*}
where $\lambda_1,...,\lambda_n$ are the eigenvalues of $A$. 
Wigner was the first to determine the limiting behavior of such a distribution when the matrix $A$ is random~\cite{Wig55}. 
He initially considered symmetric matrices with $0$'s on the diagonal and independent plus or minus $1$'s in the upper-triangle
and showed that when scaled by $\frac{1}{\sqrt{n}}$ the empirical distribution  
converges in probability to the \emph{Semicircular Law}
\begin{equation*}
f_W(x)= \Big\{
\begin{array}{cl}
\frac{1}{2\pi}\sqrt{4-x^2}& \textnormal{when }\;|x|\leq 2\\
0&\textnormal{otherwise }\hspace{1cm}.
\end{array}
\end{equation*}
Wigner later proved that the same limiting distribution holds for symmetric random variables with finite moments~\cite{Wig58}. 

The second classical type of random matrix is the Wishart matrix~\cite{Wis28}. 
Let $H\in \C^{m\times n}$ have independent Gaussian entries with variance $\frac{1}{n}$. 
Then $HH^*$ is called a Wishart matrix. 
Mar\v cenko and Pastur showed that the empirical distribution of $HH^*$ converges in probability to 
\begin{equation*}
f_{HH^*}(x)=\max(0,1-\frac{n}{m})\delta(x) +\frac{\sqrt{(x-c_-)(c_+-x)}}{2\pi x}\cdot I_{[c_,c_+]},
\end{equation*}
where $c_\pm =(1-\sqrt{\frac{m}{n}})^2$ \cite{MP67}. 
Independently, Silverstein and Grenander used a similar technique and proved almost sure convergence~\cite{GS77}. 

This paper applies the approach of Mar\v cenko and Pastur to a question originating in geometric functional analysis: 
we address the singular values of random submatrices of random unitary matrices with Haar distribution and of discrete 
Fourier transform (DFT) matrices when a random subset of columns and rows is removed 
and determine their limiting empirical singular value distribution. 
Let $\Ubn$ denote the $n\times n$ random unitary matrix with Haar distribution, and $\Un$ a realization of this random matrix. 
Let $T$ and $\Omega$ be subsets of $\{1,...,n\}$. 
We define $\UOT$ to be the matrix obtained from $\Un$ by removing rows with indices not in $\Omega$ and columns with indices not in $T$.
The $n\times n$ DFT matrix has entries
\begin{equation*}
F_{jk}=\frac{1}{\sqrt{n}}e^{-2\pi i(j-1)(k-1)/n}.
\end{equation*}
$F$ is unitary, and we define $\FOT$ analogously to $\UOT$. 

We show that when each index is included in $\Omega$ independently with probability $(1-q)$ and in $T$ independently with probability $(1-p)$, 
then the limiting empirical distribution of $\UbOT\UbOT^*$ and $ \FOT\FOT$ depends only on the parameters $p$ and $q$ and converges 
almost surely to 
\begin{small}\begin{equation*}
f_{p,q}(x)=\frac{\sqrt{(1-\frac{r_{-}}{x})(\frac{r_{+}}{x}-1)}}{2\pi(1-x)(1-\max(p,q))}\cdot I_{(r_{-},r_{+})}(x)+\frac{\max(0,1-(p+q))}{1-\max(p,q)}\cdot\delta(x-1),\label{thedistr}
\end{equation*}\end{small}
where
\begin{equation*}
r_{-}=(\sqrt{p(1-q)}-\sqrt{q(1-p)})^2
\end{equation*}
and
\begin{equation*}
r_{+}=(\sqrt{p(1-q)}+\sqrt{q(1-p)})^2.
\end{equation*}
This is formally stated as Theorem~\ref{first}.

The eigenvalue distribution of random unitary matrices with a fixed proportion of 
the bottom rows and right-most columns removed has already been studied. 
In the case when these proportions are equal, i.e. when the resulting matrix is square, 
the limiting empirical  eigenvalue density was derived in~\cite{ZS00}, 
which builds on the work in~\cite{PR05}.


The interest in the spectrum of these matrices from the perspective of geometric functional analysis is in the largest and smallest eigenvalues of 
$\UbOT\UbOT^*$ and $ \FOT\FOT$, initially perhaps asymptotically, but ideally for finite dimension. 
These eigenvalues are related to discrete uncertainty principles, as well as random projections and embeddings. 
More discussion of their significance is given following Theorem~\ref{first}. 
The first works on the extremal eigenvalues in the Wishart case were \cite{Gem80,YBK88,BY93}, 
and in the Wigner case it was~\cite{BY88}. 
It is important to note that the limiting empirical distributions were determined in these two cases 
before the behavior of the extremal eigenvalues was proved and were essential in that effort. 
We hope that the distribution presented here leads to similar developments. 


\begin{subsection}{Notation}
To make notation easiest, a single subscript will denote the dimension of a square matrix, 
while a double index will refer to an entry of the matrix. 
Thus $\Fn$ will denote the $n$-dimensional DFT matrix, 
and 
\begin{equation*}
F_{jk}=\frac{1}{\sqrt{n}}e^{-2\pi i(j-1)(k-1)/n}
\end{equation*}
will denote its entry at index $(j,k)$. 
$\Ubn$ will denote the random unitary matrix with Haar distribution and of dimension $n$ and 
$\Un$ will denote a realization of this random variable.

When we want to make the original dimension apparent, we write $\FOnTn$. 
$\UbOT$ and $\UbOnTn$ denote the analogous random variables for the Haar case. 
We will use $V$ to denote an arbitrary unitary matrix. 

We find it helpful to also work with matrices with rows and columns set to zero rather than removed. 
For clarity we make the following definitions.
\begin{definition}
A square matrix is called a \emph{diagonal projection matrix} 
if its off-diagonal entries are all zero and its diagonal entries are 
zero or one.
\end{definition}
\begin{definition}
A random diagonal projection matrix will be called a \emph{Bernoulli} diagonal projection 
matrix if the diagonal entries are independent and equal to $1$ with probability $1-p$ and 
equal to $0$ with probability $p$. 
\end{definition}

The matrices $\Pn$ and $\Qn$ will denote independent Bernoulli 
diagonal projection matrices. 
Asymptotically, $\Pn$ and $\Qn$ randomly ``erase'' the percentage $p$ and $q$ respectively of a vector.
For a matrix $A$, $A^{*}$ denotes the conjugate transpose of $A$.

\begin{note}
Throughout this paper we take the square root of a complex number 
to be uniquely defined by having argument in $[0,\pi)$. 
The reader will see that this is justified.
\end{note} 

\begin{note}\label{notoz}
When either $p$ or $q$ is $0$ or $1$, the corresponding matrix is trivial. 
For the convergence of several sums in later proofs, we assume that $p,q\in(0,1)$.
\end{note} 

\end{subsection}

\end{section}

\begin{section}{The Stieltjes and $\eta$ Transforms}
Our main tool is the Stieltjes transform, which is only defined for real random variables. 
Thus, we will determine the limiting eigenvalue distribution of 
\begin{equation*}
\Pn\Fn\Qn\Fnstar\Pn,\label{matrix}
\end{equation*}
which of course is real and contained in  $[0,1]$. 

The Stieltjes transform of a real random variable $X$ with distribution function $F_{X}(x)$ 
is a function 
$m_{X}:\C^{+}\rightarrow \R$ defined by 
\begin{equation*}
m_{X}(z)=\E_{X}\left[\frac{1}{X-z}\right].
\end{equation*}
If $F_X$ is continuous at $x$, then $f_X(x)$ can be recovered by the Stieltjes inversion formula~\cite{AGZ09} 
\begin{equation}
f_{X}(x)=\frac{1}{\pi}\lim_{\omega\rightarrow 0} \Im\; m_{X}(x+i\omega)\label{invstieltjes}.
\end{equation}
We will determine the Stieltjes transforms of $\Pn\Fn\Qn\Fnstar\Pn$ and $\Pn\Ubn\Qn\Ubnstar\Pn$ by first using the 
$\eta$-transform, which was introduced by Tulino and Verd\'u in~\cite{TV04}.
For a real valued random variable $X$, the $\eta$-transform is also a 
function 
$\eta_{X}:\C^{+}\rightarrow \R$ defined by 
\begin{equation*}
\eta_{X}(z)=\E_{X}\left[\frac{1}{1+zX}\right].
\end{equation*}
Note that for $z$ in an appropriate region of convergence
\begin{equation*}
m_{X}(z)=-\frac{1}{z}\sum_{k=0}^{\infty}(z)^{-k}\E[X^{k}]
\end{equation*}
and 
\begin{equation}
\eta_{X}(z)=\sum_{k=0}^{\infty}(-z)^{k}\E[X^{k}]\label{sumforeta},
\end{equation}
so that 
\begin{equation}
m_{X}(z)=-\frac{1}{z}\eta_{X}(-\frac{1}{z})\label{transforms}.
\end{equation}

In this section we determine the $\eta$-transform for the matrices 
$\Pn\Fn\Qn\Fnstar\Pn$ and $\Pn\Ubn\Qn\Ubnstar\Pn$, Proposition~\ref{isequiv}. 
We require several lemmas en route to this proposition. 
\begin{lemma}\label{ld}
Let $\Pb$ be a mean-zero, random diagonal matrix with independent entries in $[-1,1]$ and of dimension $n$. 
Then there exists a constant $C_m$ such that
for dimension $n$ and all $1\leq i,j\leq n$,
\begin{equation*}
 \E |F^*\Pb F|_{i,j}^m\leq  C_m  n^{-m/2}. 
\end{equation*}
The constant $C_m$ increases with $m$. 
\end{lemma}

\begin{proof}
Since $F^*\Pb F$ is Toeplitz we look at $(F^*\Pb F)_{0,l}=\frac{1}{n}\sum_{k=1}^{n} \Pb_{k,k}e^{2\pi i (k-1)(l-1)/n}$. 
Set $Y_k=\Re \Pb_{k,k}e^{2\pi i (k-1)(l-1)/n}$ and $Z_k= \Im\Pb_{k,k}e^{2\pi i (k-1)(l-1)/n}$. 
The $Y_k$, $k=1,...,n$, are independent random variables,
and $-1\leq Y_k\leq 1$ for all $k$. 
The same also holds for the $Z_k$. 
Since the $ \Pb_k$ are mean-zero, $\E\sum_{k=1}^nY_k =0$ and 
$\E\sum_{k=1}^nZ_k =0$. 
We have 
\begin{eqnarray*}
\P(|\frac{1}{n}\sum_{k=1}^{n}(Y_k+iZ_k)| > t)&\leq& \P(|\frac{1}{n}\sum_{k=1}^{n}Y_k|>\frac{t}{\sqrt{2}}\;\cup\; |\frac{1}{n}\sum_{k=1}^{n}Z_k|>\frac{t}{\sqrt{2}} )\\
&\leq & \P(|\frac{1}{n}\sum_{k=1}^{n}Y_k|>\frac{t}{\sqrt{2}})+\P( |\frac{1}{n}\sum_{k=1}^{n}Z_k|>\frac{t}{\sqrt{2}} )\\
&\leq &  4e^{-t^2 \frac{n}{4}},
\end{eqnarray*}
where the last inequality is Hoeffding's inequality.

Lastly, 
\begin{eqnarray*}
\E |\frac{1}{n}\sum_{k=1}^{n}(Y_k+iZ_k)|^m&=&\int_0^\infty ms^{m-1}\P (|\frac{1}{n}\sum_{k=1}^{n}(Y_k+iZ_k)|\geq s)ds\\
&\leq & 4 m\int_0^\infty s^{m-1} e^{-\frac{s^2 n}{4}}ds\\
&=&4m \int_0^\infty2^{m-1} \left(\frac{t}{n}\right)^{(m-1)/2}e^{-t}(nt)^{-1/2}dt\\
&=&m 2^{m+1}n^{-m/2}\int_0^\infty t^{\frac{m}{2}-1}e^{-t}dt\\
&=&m2^{m+1}n^{-m/2}\Gamma(\frac{m}{2}-1).
\end{eqnarray*}
\end{proof}

For the Haar case we use the work of Garnaev and Gluskin~\cite{GG84,Glu84}. 
The probability given in Theorem~3 in \cite{Glu84} combined with the main theorem of \cite{GG84} 
give the following theorem. 
Here $B^m_p$ denotes the unit ball in $l^p(\R^m)$, and $P_H$ denotes the orthogonal projection onto 
the subspace $H$. 
\begin{theorem}[Main Theorem in \cite{GG84}]\label{GG}
Let $H$ be an $m$-dimensional subspace of $\R^n$, $m<n$, with Grassmanian distribution. 
Then there exist absolute constants $c$ and $M$ depending only on the 
ratio of $n$ to $m$ such that with probability $1-e^{-cn}$ 
\begin{equation*}
P_HB^n_2\subset \frac{M}{\sqrt{n}}P_HB^n_\infty.\label{Kashin}
\end{equation*}
\end{theorem}


\begin{lemma}\label{ldK}
Let $\Pb$ be a mean-zero random diagonal matrix with independent entries in $[-1,1]$ and of dimension $n$.
Then there exists a constant $C_m$ such that
for dimension $n$ and all $1\leq i,j\leq n$, 
\begin{equation*}
 \E_{\Pb,\Ub} |[U^*\Pb U]_{i,j}|^m\leq C_m  n^{-m/2}.
\end{equation*}
The constant $C_m$ increases with $m$. 
\end{lemma}
\begin{proof}
We set $T_{ij}=[U^*\Pb U]_{i,j}$. 
We take $c$ to be the constant given 
by Theorem~\ref{GG} for a random $2$-dimensional subspace of a larger dimension $m_0$ for some large $m_0$. 
Since $c$ is non-decreasing as $2/n\rightarrow 0$, we may use the decay constant for an 
arbitrary large $m_0$. 
$M$ is also the constant given by Theorem~\ref{GG} for a random $2$-dimensional subspace of a larger dimension $m_0$.  
Denote by $\vec{u}_i$ the $i^{th}$ column of $U$. 
\begin{eqnarray*}
\lefteqn{\E_{\Pb,\Ub} |T_{i,j}|^m}\nonumber\\
&=&\E[|T_{i,j}|^m\big|\; \|\vec{u}_i\|_\infty\leq M \;\textnormal{ and } \|\vec{u}_j\|_\infty\leq M   ]\cdot \P\Big(\big|\; \|\vec{u}_i\|_\infty\leq M \;\textnormal{ and } \|\vec{u}_j\|_\infty\leq M        \Big)\nonumber\nonumber\\
&+&\E[|T_{i,j}|^m| \big|\; \|\vec{u}_i\|_\infty>M \;\textnormal{ or } \|\vec{u}_j\|_\infty> M  ]\cdot \P\Big(\big|\; \|\vec{u}_i\|_\infty> M \;\textnormal{ or } \|\vec{u}_j\|_\infty > M   \Big)\nonumber\\
&\leq & \E[|T_{i,j}|^m|\big|\; \|\vec{u}_i\|_\infty\leq M \;\textnormal{ and } \|\vec{u}_j\|_\infty\leq M]+e^{-cn}.\nonumber
\end{eqnarray*}
Set $Y_k=n\Re \Pb_{k,k} U^{*}_{i,k}U_{k,j}$ and $Z_k= n\Im\Pb_{k,k}U^*_{i,k}U_{k,j}$.  
The $Y_k$, $k=1,...,n$, are independent random variables,
and $-M^2\leq Y_k\leq M^2$ for all $k$. 
The same also holds for the $Z_k$. 
Since the $ \Pb_k$ are mean-zero, $\E\sum_{k=1}^nY_k =0$ and 
$\E\sum_{k=1}^nZ_k =0$. 
Now we may repeat the calculation used in the proof of Lemma~\ref{ld} keeping in mind the factor $M^2$. 
We incorporate $e^{-cn}$ into $Cn^{-m/2}$. 
\end{proof}

We again use $V$ to 
represent either a DFT matrix or a random unitary matrix.  
We then define the matrix $\Wn=\Pn\Vn$; however, in what follows 
we will not write the subscript $n$. 
We denote the $i^{th}$ column of $\Wn$ denoted $\wi$, 
and define the following quantity. 
\begin{eqnarray}
A_{i}&=&I+zPVQ\Vstar P-zQ_{i,i}\wi\wi^*\label{defAi}\\
&=&I+z\sum_{j\neq i}Q_{j,j}\wj\wj^*\label{defA}.
\end{eqnarray}

\begin{lemma}\label{convergence}
In both the DFT and the Haar cases, for $|z|<1$ the random variable $\wi^*\Aii\wi$ defined in equation~(\ref{defA}) equals a deterministic constant $D(z) $ independent of the dimension $n$ plus a 
random part that depends on the dimension and that converges almost surely to $0$ with respect to dimension 
independent of the index $i$. That is, in dimension $n$, $\wi^*\Aii\wi=D(z)+X_{n,i}$, where $D(z)$ is independent of $n$ and 
\begin{equation*}
\P(|X_{n,i}|>\epsilon)= \mathcal{O}(\epsilon^{-4}n^{-2}). 
\end{equation*}
\end{lemma}

\begin{proof}
We arbitrarily select an index $i$ and denote it $i^*$. 
If $|z|<1$, then for any realization of $P_n$ and $Q_n$ every entry of the 
following sum converges:
\begin{eqnarray*}
W^*A^{-1}_{i^*} W&=&V^{*}P(I+zPVQV^{*}P-zQ_{\iistar} \wis\wiss)^{-1}PV\\
&=&V^{*}P  \sum_{k=0}^{\infty}(-z)^k  (PVQV^{*}P-zQ_{\iistar} \wis\wiss)^k PV.
\end{eqnarray*}
For $|z|<1$, for any $\delta$ we may choose $K$ such that 
\begin{equation*}
\big|[V^{*}P  \sum_{k=K+1}^{\infty}(-z)^k  (PVQV^{*}P)^k PV]_{\iistar}\big|<\delta
\end{equation*}
for any realization of $Q$, independent of $n$. 
For now we just take $K$ to be a large integer. 
Observe that equation~(\ref{defAi}) is equivalent to 
requiring the random matrix $Q$ to have a deterministic zero at  the entry $\iistar$. 
Therefore we set $\Qt_{i,i}=Q_{i,i}-(1-q)$ for $i\neq i^*$ and $\Qt_{i^*,i^*}=-(1-q)$.
Note that $PV(1-q)I\Vstar P=(1-q) P$. 
For a fixed $K$ we now consider 
\begin{eqnarray*}
\lefteqn{V^{*}P  \sum_{k=0}^{K}(-z)^k  (PVQV^{*}P-zQ_{\iistar} \wis\wiss)^k PV}\\
&=&\sum_{k=0}^{K}(-z)^{k}\Vstar P[PV\Qt\Vstar P+(1-q)P]^{k}PV\label{sumtoK}\\
&=&\sum_{k=0}^{K}(-z)^{k}\sum_{j=0}^{k}\binom{k}{j}(1-q)^{j}\Vstar P[PV\Qt\Vstar P]^{k-j}P^{j} P V\nonumber\\
&=&\sum_{k=0}^{K}(-z)^{k}\sum_{j=0}^{k}\binom{k}{j}(1-q)^{j}\Vstar P[V\Qt\Vstar P]^{k-j} V.\nonumber\label{dec20a}
\end{eqnarray*}
Now we center the $P$ matrices. 
Set $\Pt=P-(1-p)I$. 
Then 
\begin{eqnarray*}
\Vstar P[V\Qt\Vstar P]^{k} V
&=& \Vstar (\Pt+(1-p)I)[V\Qt\Vstar (\Pt+(1-p)I)]^{k} V\\
&=&\sum_\alpha(1-p)^{k+1-|\alpha|}\Vstar \Pt^{\alpha_{0}}V\Qt \Vstar  \Pt^{\alpha_{1}} V\Qt\cdots\Pt^{\alpha_k}V\label{aaaa}
\end{eqnarray*}
for $\alpha_0,...,\alpha_k$ equaling $0$ or $1$ and $|\alpha|$ equalling the sum of the $\alpha_i$'s.
In the case that all $\alpha_i$'s are $0$, we recall that $\Qtistar=-(1-q)$, and thus  
the $\iistar$ entry of~(\ref{aaaa}) is deterministic in that case. 
Thus, the $\iistar$ entry of~(\ref{aaaa}) equals a constant independent of $n$ plus a linear combination 
of the $\iistar$ entries of matrices of the form 
\begin{equation*}
\Vstar \Pt^{\alpha_{0}}V\Qt \Vstar  \Pt^{\alpha_{1}} V\Qt\cdots\Pt^{\alpha_k}V
\end{equation*}
with $\alpha_i\neq 0$ for at least one $i$. 
We continue to center each random diagonal matrix in this way such that eventually we only 
have constant terms, independent of $n$, and terms of the form 
\begin{equation}
\Vstar \Pb^{(1)} V\Qb^{(1)}\cdots \Vstar \Pb^{(k)} V\Qb^{(k)}\Vstar  \Pb^{(k+1)} V\label{dec20b}
\end{equation}
for some centered matrices (except for the $\iistar$-entry) $\Pb^{(1)},....,\Pb^{(k+1)} $ 
and $\Qb^{(1)}... \Qb^{(k)}$ and some $1<k\leq K$. 
Note that the dimension $n$ plays no role in these expansions.  
Thus the term in equation~(\ref{dec20a}) has a deterministic part independent of the dimension 
and a random part that is a sum of terms of the form~(\ref{dec20b}). 
The number of such terms depends only on $K$; call this quantity $K_1$.

We set $T^{(i)}=\Vstar \Pb^{(i)} V$ and consider a term of the form~(\ref{dec20b}). 
\begin{eqnarray}
&&\E|(T^{(1)}  \Qb^{(1)}\cdots  T^{(k)}  \Qb^{(k)}T^{(k+1)})_{\iistar}|^4\label{mom4}\\
&&\hspace{.5cm}=\sum_{l,m,r,s}\{\E \Qb^{(1)}_{l_{1}}\cdots \Qb^{(k)}_{l_{k}}\Qb^{(1)}_{m_{1}}\cdots \Qb^{(k)}_{m_{k}}\Qb^{(1)}_{r_{1}}\cdots \Qb^{(k)}_{r_{k}}\Qb^{(1)}_{s_{1}}\cdots \Qb^{(k)}_{s_{k}}\}\nonumber\\
&&\hspace{.5in}\times\E T^{(1)}_{i^*,l_{1}}\cdots T^{(k+1)}_{l_{k},i^*}T^{(1)}_{i^*,m_{1}}\cdots T^{(k+1)}_{m_{k},i^*}T^{(1)}_{i^*,r_{1}}\cdots T^{(k+1)}_{r_{k},i^*}T^{(1)}_{i^*,s_{1}}\cdots T^{(k+1)}_{s_{k},i^*}\}\nonumber .
\end{eqnarray}

By Lemma~\ref{ld}, $\E_{Pn,Q_n} |T^{(l)}_{i,j}|^m\leq C_m n^{-(m+1)/2}$ in the DFT case 
for $1\leq i,j\leq n$. 
Lemma~\ref{ldK} delivers the same bound for  $\E_{\Ubn,Pn,Q_n} |T^{(l)}_{i,j}|^m$ in the 
Haar case. 

\begin{eqnarray*}
\lefteqn{\E|T^{(1)}_{i^*,l_{1}}...T^{(k+1)}_{l_{k},i^*}T^{(1)}_{i^*,m_{1}}...T^{(k+1)}_{m_{k},i^*}T^{(1)}_{i^*,r_{1}}...T^{(k+1)}_{r_{k},i^*}T^{(1)}_{i^*,s_{1}}...T^{(k+1)}_{s_{k},i^*}|}\\
&\leq &(\E|T^{(1)}_{i^*,l_{1}}|^{4(k+1)}\cdots\E|T^{(k)}_{l_{k},i^*}|^{4(k+1)}\cdots\cdots\E|T^{(1)}_{i^*,s_{1}}|^{4(k+1)}\cdots\E|T^{(k)}_{s_{k},i^*}|^{4(k+1)})^{1/(4(k+1))}\nonumber\\
&\leq& C_{4(k+1)}n^{-(4(k+1))/2}\nonumber\\
&\leq&  C_{K_1}n^{-(2k+2)}.
\end{eqnarray*}
We now  bound 
\begin{equation}
\sum_{l,m,r,s}\E \Qb^{(1)}_{l_{1}}\cdots \Qb^{(k)}_{l_{k}}\Qb^{(1)}_{m_{1}}\cdots \Qb^{(k)}_{m_{k}}\Qb^{(1)}_{r_{1}}\cdots \Qb^{(k)}_{r_{k}}\Qb^{(1)}_{s_{1}}\cdots \Qb^{(k)}_{s_{k}}.\label{qq}
\end{equation}
The expectations in line~(\ref{qq}) are all less 
than or equal $1$, so the summability is solely a question of how many terms in the sums there are. 
Since the $\Qb^{(i)}$'s are independent and centered, the expectation of a product of 
$\Qb_{i}$'s is zero if there is not at least the square of each term or the non-zero term $\Qb^{(i)}_{\iistar}$. 
Regardless of whether $i^*$ is an index in an expectation, the number of possible other indices is at 
most $2k$, 
and for $j=1,...,2k$, there are nonzero expectations for $j$ terms different from $i^*$. 
Once $j$ integers out of $\{1,...,n\}$  are chosen, the number of ways to assign them to $4k$ positions is 
independent of $n$. 
Call this number $C_{j,4k}$. 
The number of ways to choose $j$ different numbers out of $n$ is $\binom{n}{j}\leq n^{j}$.
Set $C'_k=2\max_{1\leq j\leq 2k}C_{j,2k}$. 
Then the term in line~(\ref{qq}) is less than or equal to
\begin{eqnarray*}
  \sum_{j=1}^{2k}C_{j,4k}n^{j}&\leq &2k\max_{1\leq j\leq 2k}C_{j,4k} n^{2k}\\
&\leq &K_1C'_{K_1}n^{2k}.
\end{eqnarray*}
Then 
\begin{eqnarray}
\lefteqn{\P(\big|\textnormal{ sum of all terms of the form~(\ref{dec20b})}\big|>\epsilon)}\hspace{4cm}\nonumber\\
&\leq& K_1\cdot \P(| \textnormal{ one such term }|>\frac{\epsilon}{K_1})\nonumber\\
&=&  K_1\cdot \P\left(| \textnormal{ one such term }|^4>\left(\frac{\epsilon}{K_1}\right)^4\right)\nonumber\\
&\leq &K_1 \cdot ( K_1^4\epsilon^{-4}  C_{K_1}n^{-(2k+2)})\cdot ( K_1 C'_{K_1}n^{2k})\nonumber\\
&=& K_1^6\epsilon^{-4}  C_{K_1}  C'_{K_1}n^{-2}    .\label{boundonprob}
\end{eqnarray}
Since the terms~(\ref{boundonprob}) are summable with respect to $n$, 
the Borel-Cantelli lemma implies that the sum of the random terms 
\begin{equation*}
\sum_{l\in \{1,...,n\}^k} T^{(1)}_{i,l_{1}}\Qb^{(1)}_{l_{1}}\cdots\Qb^{(k)}_{l_{k}}T^{(k+1)}_{l_{k+1},i}   
\end{equation*}
converges almost surely to $0$ for all $k$ as $n \rightarrow \infty$. 

This argument is independent of which index we denote $i^*$, and so the convergence 
is almost sure to the same constant independent of the placement of the zero. 

By taking the limit as the dimension implicit in line~(\ref{sumtoK}) 
tends to infinity, we obtain the constant $D_K(z)$. 
Then $\{D_k(z)\}_{k=1}^\infty$ is a Cauchy sequence, and, if we denote its limit $D(z)$, then we have 
$|D_k(z)-D(z)|\leq \frac{|z|^k}{1-|z|}$. 

To show the almost sure convergence we let $\epsilon>0$ be given. 
Choose $K$ large enough so that $ \frac{|z|^K}{1-|z|}\leq \epsilon/4$ and 
$|D_K(z)-D(z)|\leq \epsilon/4$. 
In the following calculation the matrices are of dimension $n$, as indexed by the sum. 
Using (\ref{boundonprob}) we have 
\begin{eqnarray*}
\lefteqn{\sum_{n=1}^\infty \P(|[W^{*}\Aii W]_{ii}-D(z)|>\epsilon)}\hspace{2cm}\\
&=& \sum_{n=1}^\infty \P(|[V^{*}P\sum_{k=0}^{\infty}(-z)^k  (PVQV^{*}P)^k PV  ]_{ii}-D(z)|>\epsilon)\\
&\leq & \sum_{n=1}^{\infty} \P(|[V^{*}P\sum_{k=0}^{K}(-z)^k  (PVQV^{*}P)^k PV  ]_{ii}-D_K(z)|\\
&&+|D_K(z)-D(z)| +|[V^{*}P\sum_{k=K+1}^{\infty}(-z)^k  (PVQV^{*}P)^k PV  ]_{ii}   |>\epsilon )\\
&\leq & \sum_{n=1}^{\infty} \P(|[V^{*}P\sum_{k=0}^{K}(-z)^k  (PVQV^{*}P)^k PV  ]_{ii}-D_K(z)|>\epsilon/2)\\
&\leq & \sum_{n=1}^{\infty}2^{4} K_1^6\epsilon^{-4}  C_{K_1}  C'_{K_1}n^{-2} ,  
\end{eqnarray*}
which is finite. 
The Borel-Cantelli lemma now gives the almost sure convergence of $[W^*\Aii W]_{ii}$ 
to $D(z)$ for all $i$. 
\end{proof}

The following lemma is essentially due to Tulino, Verd\'u, Caire and Shamai, 
and was developed in their work on (deterministic) Toeplitz matrices conjugated 
by a random Bernoulli projection matrix, which they call 'erasure matrices'~\cite{TVCS07}. 
The manipulations and the insight concerning the term $[W^*\Aii W]_{ii}$ are theirs. 
However, $[W^*\Aii W]_{ii}$ has a different form in the work presented here.   
As a consequence, the proof of Lemma~\ref{implicit} requires 
the preceding lemmas, which are our own. 
The proof given here is also self-contained. 
Equation~(\ref{fixed}) certainly holds more broadly than just the case 
covered in~\cite{TVCS07} and the work presented here. 
Similar general settings where such equations hold are proved in~\cite{TCSV10}.

\begin{lemma}\label{implicit}
The $\eta$-transforms of $\Pn\Fn\Qn\Fnstar\Pn$ and $\Pn\Ubn\Qn\Ubnstar\Pn$ converge almost surely to the 
same function, which we denote $\etapq$. 
This function is a solution to the 
implicit equation 
\begin{equation}
\etapq(z)=\eta_{Q}\left(z-z\frac{p}{\etapq(z)}\right)\label{fixed},
\end{equation}
where $\eta_Q$ is the asymptotic $\eta$-transform of $\Qn$. 
\end{lemma}

\begin{proof}
We return to the set-up given in equations~(\ref{defAi}) and~(\ref{defA}).
We again at first let $V$ stand for an arbitrary unitary matrix.  
Recall that $\Wn=\Pn\Vn$ and that  $\wi$ is the $i^{th}$ column of $\Wn$. 
As defined in (\ref{defAi}) and~(\ref{defA})
\begin{eqnarray*}
A_{i}&=&I+zPVQ\Vstar P-zQ_{i,i}\wi\wi^*\\
&=&I+z\sum_{j\neq i}Q_{j,j}\wj\wjs.
\end{eqnarray*}
$A_{i}$ is invertible for $z\notin [-1,0]$, and one can verify directly that 
\begin{equation}
(I+zPVQ\Vstar P)^{-1}=A_{i}^{-1}-\frac{zQ_{i,i}}{1+zQ_{i,i}\wis\Aii\wi}\Aii\wi\wis\Aii.\label{68}
\end{equation}
Now we multiply both sides of equation~(\ref{68}) by $zQ_{i,i}\wi\wis$ and obtain
\begin{eqnarray*}
\lefteqn{zQ_{i,i}\wi\wis(I+zPVQ\Vstar P)^{-1}}\\
&=&zQ_{i,i}\wi\wis\Aii-\frac{z^2Q^2_{i,i}\wi\wis}{1+zQ_{i,i}\wis\Aii\wi}\Aii\wi\wis\Aii\nonumber\\
&=&zQ_{i,i}\wi\wis\Aii(1-\frac{zQ_{i,i}\wis\Aii\wi}{1+zQ_{i,i}\wis\Aii\wi})\\
&=&\frac{zQ_{i,i}}{1+zQ_{i,i}\wis\Aii\wi}\wi\wis\Aii.
\end{eqnarray*}
Summing over $i$ gives
\begin{eqnarray}
z PVQ\Vstar P (I+zPVQ\Vstar P)^{-1}&=&\sum_{i=1}^{n}zQ_{i,i}\wi\wis (I+zPVQ\Vstar P)^{-1}\\
&=&\sum_{i=1}^{n}\frac{zQ_{i,i}}{1+zQ_{i,i}\wis\Aii\wi}\wi\wis\Aii\label{73}.
\end{eqnarray}
We use the following observation: let $M\in\C^{n\times n}$ be a positive matrix and 
$\lambda_{1},...,\lambda_{n}$ its eigenvalues. 
Then 
\begin{eqnarray*}
\tr((I+M)^{-1})&=&\frac{1}{n}\sum_{i=1}^{n}\frac{1}{1+\lambda_{i}}\\
&=& \frac{1}{n}\sum_{i=1}^{n}\frac{1+\lambda_{i}}{1+\lambda_{i}}-\frac{1}{n}\sum_{i=1}^{n}\frac{\lambda_{i}}{1+\lambda_{i}}\\
&=&1-\tr (M(I+M)^{-1}).
\end{eqnarray*}
Similarly, using equation~(\ref{73}),
\begin{eqnarray*}
\tr((I+zPVQ\Vstar P)^{-1})&=&1-\tr (zPVQ\Vstar P (I+zPVQ\Vstar P)^{-1})\\
\tr((I+zPVQ\Vstar P)^{-1})&=&1-\tr (zQ\Vstar P (I+zPVQ\Vstar P)^{-1}PV^*)\\
&=&1- \tr (\sum_{i=1}^{n}\frac{zQ_{i,i}\wi\wis\Aii}{1+zQ_{i,i}\wis\Aii\wi})\\
&=&1-\frac{1}{n}\sum_{i=1}^{n}\frac{zQ_{i,i}\wis\Aii\wi}{1+zQ_{i,i}\wis\Aii\wi}\\
&=&\frac{1}{n}\sum_{i=1}^{n}\frac{1}{1+zQ_{i,i}\wis\Aii\wi}.
\end{eqnarray*}
We note
\begin{eqnarray}
\lefteqn{\tr(\Vstar P(I+zPVQ\Vstar P)^{-1}PV)}\nonumber\\
&=&\tr ( P(I+zPVQ\Vstar P)^{-1}P)\nonumber\\
&=&\tr((I+zPVQ\Vstar P)^{-1})-\frac{1}{n}\sum_{i=1}^{n}1\{P_{i,i}=0\}[(I+zPVQ\Vstar P)^{-1}]_{i,i}\nonumber\\
&=&\tr((I+zPVQ\Vstar P)^{-1})-\frac{1}{n}\sum_{i=1}^{n}1\{P_{i,i}=0\},\label{87}
\end{eqnarray}
where the last line follows from using that $[(I+zPVQ\Vstar P)^{-1}]_{i,i}=1$ when $P_{i,i}=0$.

Lemma~\ref{convergence} 
states that when $V_n=F_n$ or when it has Haar distribution $\wi^*\Aii\wi$ 
converges almost surely to a number $D(z) $ independent of $i$ as $n\rightarrow \infty$. 
This constant is the same for both the Fourier and the Haar case. 
Writing $\E$ for both $\E_{\Pn,\Qn}$ and $\E_{\Pn,\Qn,\Ubn}$, 
\begin{eqnarray*}
\etapq(z)&=&\lim_{n\rightarrow \infty}\E\;\tr((I+zPVQ\Vstar P)^{-1})\\
&=&\lim_{n\rightarrow\infty}\E\; \frac{1}{n}\sum_{i=1}^{n}\frac{1}{1+zQ_{i,i}\wis\Aii\wi}.
\end{eqnarray*}
Let $C_z$ be the decay constant given in Lemma~\ref{convergence}. 
For each $i$ we write $\wis\Aii\wi=D(z)+X_{n,i}$, where $D(z)$ is the deterministic part independent of dimension, as in Lemma~\ref{convergence}. 
\begin{eqnarray}
\lefteqn{\P (|  \frac{1}{n}\sum_{i=1}^{n}\frac{1}{1+zQ_{i,i}(D(z)+X_{n,i} )}-\eta_Q(zD(z))|>\epsilon)}\hspace{2cm}\nonumber\\
&\leq& \P (|\frac{1}{n}\sum_{i=1}^{n}\frac{1}{1+zQ_{i,i} (D(z)+X_{n,i} )}-\frac{1}{n}\sum_{i=1}^{n} \frac{1}{1+zQ_{i,i}D(z)}|\nonumber\\
&&\hspace{2cm}+|\frac{1}{n}\sum_{i=1}^{n} \frac{1}{1+zQ_{i,i}D(z)}-\eta_Q(D(z)z)|>\epsilon)\nonumber\\
&\leq & \P (|\frac{1}{n}\sum_{i=1}^{n}\frac{1}{1+zQ_{i,i}(D(z)+X_{n,i})}-\frac{1}{n}\sum_{i=1}^{n} \frac{1}{1+zQ_{i,i}D(z)}|>\frac{\epsilon}{2})\nonumber\\
&&\hspace{2cm}+ \P (|\frac{1}{n}\sum_{i=1}^{n} \frac{1}{1+zQ_{i,i}D(z)}-\eta_Q(D(z)z)|>\frac{\epsilon}{2})\label{bernstein1}\\
&\leq & 2 \P (\frac{1}{n}\sum_{i=1}^{n}|\frac{zQ_{i,i}X_{n,i}}{(1+zQ_{i,i}D(z))( 1+zQ_{i,i}(D(z)+X_{n,i})) } | >\frac{\epsilon}{2})\label{bernstein}\\
&\leq & 2^5 \epsilon^{-4}  \frac{1}{n}\sum_{i=1}^{n}\E|\frac{zQ_{i,i}X_{n,i}}{(1+zQ_{i,i}D(z))( 1+zQ_{i,i}(D(z)+X_{n,i})) }    |^4 \\
&\leq & 2^5 C_z C_d C_r n^{-2},\label{lastline}
\end{eqnarray}
where we use the  exponential Bernstein bound for the term in line~(\ref{bernstein1}) 
and incorporate it into the $2$ in line~(\ref{bernstein}). 
For the last inequality~(\ref{lastline}) we use $C_d$ to take care of the denominator, 
which is possible for $z$ in a small enough circle around the origin. 
As in Lemma~\ref{convergence}, we truncate the infinite sum implicit in each $X_{n,i}$  
at some index $K$ and collect the remainder in the term $C_r$. 
The remaining terms in the numerator are all products of terms of the form~(\ref{mom4}) 
with powers summing to $4$, and so the work of Lemma~\ref{convergence} applies. 
We use Lemma~\ref{convergence} for the final  inequality. 
Using the Borel-Cantelli lemma, we now have the almost sure convergence 
\begin{equation*}
 \frac{1}{n}\sum_{i=1}^{n}\frac{1}{1+zQ_{i,i}\wis\Aii\wi}\stackrel{a.s}{\rightarrow } \eta_Q(zD(z)).
\end{equation*}
Using this and Lemma~\ref{convergence} again,  
\begin{eqnarray}
D(z)  \etapq(z)&=&\lim_{n\rightarrow\infty}\E \frac{1}{n}\sum_{i=1}^{n}\frac{\wis\Aii\wi}{1+zQ_{i,i}\wis\Aii\wi}\nonumber\\
&=&\lim_{n\rightarrow\infty}\E \frac{1}{n}\sum_{i=1}^{n}\wis(I+zWQ\Ws)^{-1}\wi\nonumber\\
&=&\lim_{n\rightarrow\infty} \E\tr(\Ws(I+zWQ\Ws)^{-1}W)\nonumber\\
&=&\etapq(z)-p\label{86},
\end{eqnarray}
where equation~(\ref{86}) follows from taking the limit with respect to $n$ 
of the expectation of equation~(\ref{87}).
\end{proof}
\begin{proposition}\label{isequiv}
Let $\Pn$ and $\Qn$ be independent Bernoulli as defined above, 
with expected traces $1-p\in(0,1)$ and $1-q\in(0,1)$ respectively.
Then the $\eta$-transforms of $\Pn\Fn\Qn\Fnstar\Pn$ and $\Pn\Ubn\Qn\Ubnstar\Pn$ 
converge almost surely 
to the asymptotic $\eta$-transform 
\begin{equation*}
\etapq(z)=\frac{1+(p+q)z+\sqrt{1+(2(p+q)-4pq)z+((p+q)^2-4pq)z^{2}}}{2(1+z)}.
\end{equation*}
\end{proposition}

\begin{proof}
The proof is the same for both types of matrices, so we write it only for the DFT case.
The matrices $\Fn\Qn\Fnstar$ have only eigenvalues $1$ and $0$, and their limiting 
$\eta$-transform is 
\begin{equation*}
\eta_{F Q\Fstar}(z)=\eta_{Q}(z)=\frac{1-qz}{1+z}.
\end{equation*}
Applying equation~(\ref{fixed}) from Lemma~\ref{implicit} yields
\begin{equation*}
\etapq(z)=
\frac{1+q(z-z\frac{p}{\etapq(z)})}{1+(z-z\frac{p}{\etapq(z)})},
\end{equation*}
which leads to the equation
\begin{equation*}
(1+z)\etapq^{2}(z)-(1+(p+q)z)\etapq(z)+pqz=0.
\end{equation*}
This equation has the solutions 
\begin{equation*}
\frac{1+(p+q)z\pm\sqrt{(1+(p+q)z)^2-4(1+z)pqz}}{2(1+z)}.
\end{equation*}
Noting that $\etapq(0)$ must equal $1$, we choose the solution with addition.
We thus have 
\begin{equation}
\etapq(z)=\frac{1+(p+q)z+\sqrt{1+(2(p+q)-4pq)z+(p-q)^2z^{2}}}{2(1+z)}.\label{efqfe}
\end{equation} 

\end{proof}

\end{section}

\begin{section}{Limiting Empirical Distributions}

\begin{theorem}\label{first}
For $i=1,...,n$ let $i$ be contained in $\On$ independently with probability $(1-q)$ and, also independently, 
let $i$ be included in $\Tn$ with probability $(1-p)$. 
Then the empirical distributions of the $\min(|\Tn|,|\On|)$ largest eigenvalues of $\FOnTn\FOnTn^* $ and $\UbOnTn\UbOnTn^* $  converge almost surely to
\begin{small}\begin{equation}
f_{p,q}(x)=\frac{\sqrt{(1-\frac{r_{-}}{x})(\frac{r_{+}}{x}-1)}}{2\pi(1-x)(1-\max(p,q))}\cdot I_{(r_{-},r_{+})}(x)+\frac{\max(0,1-(p+q))}{1-\max(p,q)}\cdot\delta(x-1)\label{distr2}
\end{equation}\end{small}
where
\begin{equation*}
r_{-}=(\sqrt{p(1-q)}-\sqrt{q(1-p)})^2
\end{equation*}
and
\begin{equation*}
r_{+}=(\sqrt{p(1-q)}+\sqrt{q(1-p)})^2.
\end{equation*}
\end{theorem}
\begin{figure}[h]\label{plot}
\begin{center}
\includegraphics[width=13cm]{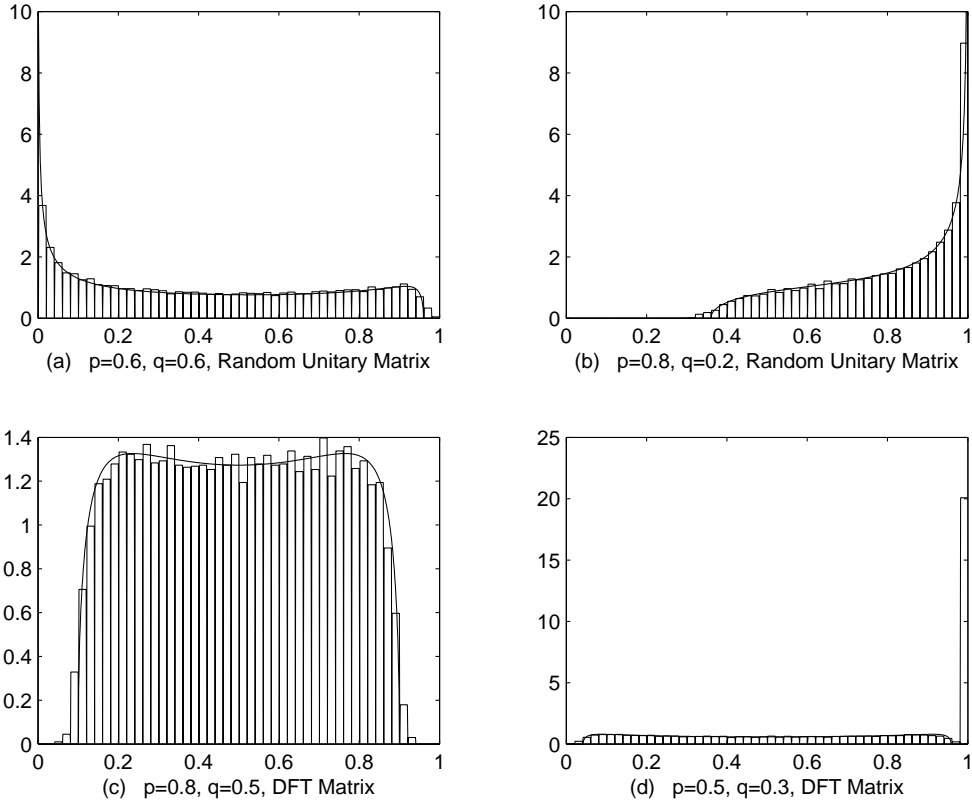}
\caption{Empirical eigenvalue distribution plotted against continuous part of asymptotic distribution. 
In plots (a) and (b) the matrices are submatrices of a random unitary matrix with Haar distribution. 
In plots (c) and (d) the matrices are submatrices of the DFT matrix. 
In each case, the original matrix 
dimension was $100/(1-\max(p,q))$, so that the submatrices all had expected dimensions $100\times 100$. Each random matrix was 
realized $100 $ times.}
\end{center}
\end{figure}
Note that $r_{-}=0$ only when $p=q$; that is, when $\FOnTn\FOnTn^* $ or $\UbOnTn\UbOnTn^* $ 
is asymptotically square. 
Therefore, when $p\neq q$ the support of the limiting distribution begins at $r_{-}>0$.
Also, $r_{+}=1$ only when $p+q=1$, 
so when  $p+q\neq 1$ there is a gap in the 
support of the limiting distribution from $r_{+}$ to $1$. 
When $r_-\neq 0$ and $r_+\neq 1$ the continuous part of the measure begins at $r_{-}$ 
and makes an arc ending at $r_{+}$. 
When $r_-=0$ or $r_+=1$, the continuous part of~(\ref{distr2}) tends to $\infty$ as $x\rightarrow 0$ or $x\rightarrow 1$. 
When $p+q<1$ there is a point mass of measure $1-(p+q)$ at $1$, 
and when $p+q>1$, the support stops at $r_{+}$ and there is no point mass at $1$.
In Figure 1 the continuous part of the asymptotic distribution is plotted against empirical values for several parameter pairs $p$ and $q$. 

We discuss the relationship between Theorem~\ref{first} and uncertainty principles and other areas  before turning to the proof. 
We focus first on the DFT case and assume that the dimension is $n$. 
The norm of a DFT matrix with a set of rows and columns removed equals $1$ if and only if there exists a vector with time support corresponding to 
the remaining columns and frequency support corresponding to the remaining rows. 
That is, denoting $Fx=\hat{x}$, in the notation of this paper, 
$\|\FOT\|=1$ if and only if there exists $x\in\C^n$ such that $\supp(x)\subset T$ and $\supp(\hat{x})\subset \Omega$. 
One is then interested in determining conditions on the cardinality of $|\Omega|$ and $|T|$ such that $\|\FOT \|<1$, 
where generally at least one set is random and a statement is made in some probabilistic form. 
This non-asymptotic question has been studied intensively over the last ten years. 
Recent results and a general discussion can be found in~\cite{Tro08}. 
While Theorem~\ref{first} does not assert the non-existence of any vectors, it 
does state when certain vectors do exist  and 
sheds light on one of the main theorems in this area, 
namely that of Tao. 
\begin{theorem}[Tao, \cite{Tao05}]
If $n$ is prime and  $|\Omega|+|T|<n$, then $\|\FOT\|<1$. 
\end{theorem}
Note that since Tao's theorem requires $n$ to be prime, it precludes the case that $\frac{|\Omega|}{n}+\frac{|T|}{n}=1$. 
Theorem~\ref{first} says that if $\frac{|\On|}{n}+\frac{|\Tn|}{n}\rightarrow 1$, 
then $\|\FOnTn\|$ converges almost surely to $1$, and in fact this also holds for random unitary matrices $\UbOnTn\UbOnTn^* $. 
Almost sure convergence and the existence of infinitely many primes imply
that for all $\epsilon_1,\epsilon_2>0$, there exists a prime $n$ and sets $\On, \Tn\subset \{1,...,n\}$ such that 
$\frac{|\On|}{n}+\frac{|\Tn|}{n} < 1+\epsilon_1$ 
and $\|\FOnTn\|>1-\epsilon_2$. 
In fact, the proportion of subsets for which $\|\FOnTn\|<1-\epsilon_2$ converges to $0$ as $n\rightarrow \infty$. 
Moreover, as $n$ increases, the Haar measure of the set of unitary matrices $U$ such that 
$\|\UOnTn\|>1-\epsilon_2$ for random sets satisfying $\frac{|\On|}{n}+\frac{|\Tn|}{n} < 1+\epsilon_1$ 
also converges to $1$. 

A further area of interest is the smallest eigenvalue of $\FOnTn \FOnTn^*$ or $\UbOnTn\UbOnTn^* $; in particular, one would like to 
bound the smallest eigenvalue away from $0$. 
See~\cite{RV08c} for recent results in the setting of independent matrix entries. 
While Theorem~\ref{first} does not make any statement on when the smallest eigenvalue is strictly positive, 
it does say that if $\frac{|\On|}{|Tn|}\rightarrow 1$, then the smallest eigenvalue of $\FOnTn \FOnTn^*$ and $\UbOnTn\UbOnTn^* $ 
converge to $0$. 
This corresponds to the behavior of square matrices with independent entries, though in that case non-asymptotic bounds away from 
$0$ exist \cite{RV08c}. 

Theorem~\ref{first} and some numerical experiments suggest the obvious conjecture that the largest and smallest eigenvalues converge to the 
edge of the limiting support. 
This would imply that Tao's result gives the general uncertainty principle behavior for DFT and random unitary matrices, and that 
the submatrices that do not have this behavior have measure zero asymptotically. 
As was the case for the Wigner and Wishart distributions, we hope that the limiting empirical distribution is helpful in 
determining the behavior of the extremal eigenvalues. 

We note, lastly, that a potential further step in this direction is 
restricted isometry properties. 
Here one set, say $T_n$, is taken at random and one seeks to bound the extremal eigenvalues of $\FOnTn \FOnTn^*$ or $\UbOnTn\UbOnTn^* $
when $\On$ ranges over all subsets of certain cardinality. 
This property of random matrices is central to compressed sensing and has received enormous attention in recent years; 
see~\cite{Can06ICM} for a recent overview.

\begin{proof}[Proof of Theorem~\ref{first}.]
We give the proof using the notation for the DFT case, but the two are identical. 
We begin by working with $ \Pn\Fn\Qn\Fnstar\Pn$, where  $\Pn$ and $\Qn$ are a sequence of independent Bernoulli
diagonal projection matrices, as defined earlier, 
with expected traces $1-p\in(0,1)$ and $1-q\in(0,1)$ respectively. 
By Proposition~\ref{isequiv} we have 
\begin{equation}
|\frac{1}{n}\sum_{i=1}^n \frac{1}{1+z\lambda_i}-\eta_{p,q}(z)|\stackrel{a.s.}{\longrightarrow}  0\label{sconv}
\end{equation}
for all $z\in \C^+$. 
By applying this to $-\frac{1}{z}$, multiplying both terms in~(\ref{sconv}) 
by $-\frac{1}{z}$ and using equation~(\ref{transforms}), we obtain 
\begin{equation*}
|\frac{1}{n}\sum_{i=1}^n \frac{1}{z-\lambda_i}-m_{p,q}(z)|\stackrel{a.s.}{\longrightarrow} 0\label{mconv}.
\end{equation*}
Thus, the random measures induced by the eigenvalues 
at each dimension $n$ converge almost surely to the probability measure corresponding to $m_{p,q}$ 
\cite{AGZ09} (Theorem 2.4.4(c) in \cite{AGZ09} also holds for almost sure convergence). 
By equation~(\ref{transforms}) the limiting Stieltjes transform is 
\begin{eqnarray*}
m_{p,q}(z)&=&\frac{-1}{z}\cdot\frac{1-(p+q)\frac{1}{z}+\sqrt{1+(2(p+q)-4pq)(\frac{-1}{z})+\b(\frac{1}{z})^{2}}}{2(1-\frac{1}{z})}\nonumber\\
&=&\frac{1-(p+q)\frac{1}{z}+\sqrt{1-\a\frac{1}{z}+\b\frac{1}{z^{2}}}}{2(1-z)},
\end{eqnarray*}
and we are interested in the inverse Stieltjes transform
\begin{equation*}
\lim_{\omega\rightarrow 0^{+}}\frac{1}{\pi}\Im\; m_{p,q}(x+i\omega).
\end{equation*}
$m_{p,q}(x+i\omega)$ is a continuous function of both $x$ and $\omega$ 
for $x\in (0,1)$. 
Thus, using equation~(\ref{invstieltjes}) for $x\in(0,1)$,
\begin{eqnarray}
f_{p,q}(x)&=&\frac{1}{\pi}\Im\;\frac{1-(p+q)\frac{1}{x}+\sqrt{1-\a\frac{1}{x}+\b\frac{1}{x^{2}}}}{2(1-x)}\nonumber\\
&=&\Im\;\frac{\sqrt{1-\a\frac{1}{x}+\b\frac{1}{x^{2}}}}{2\pi(1-x)}\label{withIm}.
\end{eqnarray}
Imitating Mar\v cenko and Pastur \cite{MP67}, 
we denote the roots of the equation $x^2-\a x+\b=0$
by $r_{-}$ and $r_{+}$. 
These values are
\begin{eqnarray*}
r_{-}&=&\frac{\a-\sqrt{\a^2-4\b}}{2}\nonumber\\
&=&\frac{\a-2\sqrt{(p+q-2pq-(p-q))(p+q-2pq+(p-q))}}{2}\nonumber\\
&=&\frac{\a-2\sqrt{(2q(1-p))(2p(1-q))}}{2}\nonumber\\
&=&p(1-q)+q(1-p)-2\sqrt{p(1-q)q(1-p)}\nonumber\\
&=&(\sqrt{p(1-q)}-\sqrt{q(1-p)})^2
\end{eqnarray*}
and $r_{+}=(\sqrt{p(1-q)}+\sqrt{q(1-p)})^2$,
as defined in the statement of the theorem.
For $x\in(0,1)$ equation~(\ref{withIm}) is now 
\begin{eqnarray}
f_{PFQ\Fstar P}(x)&=&\Im\;\frac{\sqrt{1-\a\frac{1}{x}+\b\frac{1}{x^{2}}}}{2\pi(1-x)}\nonumber\\
&=&\Im\frac{\sqrt{(1-\frac{r_{-}}{x})(1-\frac{r_{+}}{x}  )}}{2\pi(1-x)}\label{factors}\\
&=&\frac{\sqrt{(1-\frac{r_{-}}{x})(\frac{r_{+}}{x} -1 )}}{2\pi(1-x)}\cdot I_{[r_{-},r_{+}]}(x).\nonumber
\end{eqnarray}

We now determine the density at $x=0$,
for which we need to find the asymptotic proportion of zero eigenvalues of 
$\Pn\Fn\Qn\Fnstar \Pn$ or $\Pn\Ubn\Qn\Ubnstar \Pn$. 
This proportion is given by $\lim_{r\rightarrow\infty}\etapq(r)$, 
since the latter quantity gives the measure of the set ${0}$ with respect 
to the measure $f_{p,q}$. 
We have 
\begin{eqnarray}
\lim_{r\rightarrow \infty}\etapq(r)&=&\lim_{r\rightarrow\infty}\frac{1+(p+q)r+\sqrt{1+\a r +\b r^2}}{2(1+r)}\nonumber\\
&=&\frac{(p+q)+|p-q|}{2}\nonumber\\
&=&\max(p,q)\label{max}.
\end{eqnarray}

Lastly, we investigate the point $x=1$. 
We denote by $\mu(1)$ the measure of the set $\{1\}$ 
with respect to the limiting distribution. 
Let $X$ denote the random variable of the eigenvalues.
We first must address the convergence of $\sum_{k=0}^{\infty}(-z)^{k}\E X^{k}$. 
If $p+q \neq 1$ and $p\neq q$, then  $0<r_{-}$ and $r_{+}<1$, and
\begin{eqnarray*}
\E X^k &\leq &\|\EXP I_{[r_{-},r_{+}]}\|_{\infty} \cdot\int_0^{r_+}r_{+}^{k}dx+\mu(1)\nonumber\\
&=&\|\EXP I_{[r_{-},r_{+}]}\|_{\infty}\cdot r_{+}^{k+1}+\mu(1).
\end{eqnarray*}

 If $p+q\neq 1$ and $p=q$, which implies $r_{-}=0$, then we pick a small $\epsilon$ and have 
\begin{eqnarray*}
\E X^k &\leq & \int_0^{\epsilon}\epsilon^k  \EXP dx\\
&& +\|\EXP I_{[\epsilon,r_{+}]}\|_{\infty} \cdot\int_0^{r_+}r_{+}^{k}dx+\mu(1).
\end{eqnarray*}

When $p+q=1$ and when $z$ belongs to a region of convergence to be determined shortly, 
equation~(\ref{sumforeta}) implies 
\begin{eqnarray}
\etapq(z)&=&\sum_{k=0}^{\infty}(-z)^{k}\E X^{k}\label{etasum}\nonumber\\
&=&\sum_{k=0}^{\infty}(-z)^{k}\int_{r_{-}}^{r_{+}}x^{k}\EXP dx+\sum_{k=0}^{\infty}\mu(1)(-z)^{k}\label{pointsum}\nonumber\\
&=&\sum_{k=0}^{\infty}(-z)^{k}\int_{r_{-}}^{r_{+}}x^{k}\EXP dx+\frac{\mu(1)}{1+z}.
\end{eqnarray}
Since we have assumed that neither $p$ or $q$ equals $1$ or $0$, 
we have $r_{-}<1$.
We show that the sum on the left side  of equation~(\ref{pointsum}) converges for 
$|z|\leq 1/r_{-}$. 
We define   $\lfloor k\rfloor_e=\Big\{\begin{array}{ll} k&k\;\textnormal{even}\\k-1&k\;\textnormal{odd} \end{array}$. 
\begin{eqnarray}
\lefteqn{\left|\sum_{k=0}^{\infty}(-z)^{k}\int_{r_{-}}^{1}x^{k}\EXP dx\right|}\nonumber\\
&\leq& \sum_{k=0}^{2}|z|^k\E X^k+\sum_{k=3}^{\infty}|z|^{k}\int_{r_{-}}^{1}x^{k}\frac{\sqrt{\frac{1}{x}-1}}{2\pi(1-x)}dx\nonumber\\
&=&\sum_{k=0}^{2}|z|^k\E X^k+\sum_{k=3}^{\infty}|z|^{k}\int_{r_{-}}^{1}x^{k}\frac{\sqrt{\frac{1-x}{x}}}{2\pi(1-x)}dx\nonumber\\
&=&\sum_{k=0}^{2}|z|^k\E X^k+\frac{1}{2\pi}\sum_{k=3}^{\infty}|z|^{k}\int_{r_{-}}^{1}x^{k-1/2}(1-x)^{-1/2}dx\nonumber\\
&\leq&\sum_{k=0}^{2}|z|^k\E X^k+\frac{|z|}{2\pi}\sum_{k=2}^{\infty}|z|^{k}\int_{r_{-}}^{1}x^{k}(1-x)^{-1/2}dx\nonumber\\
&\leq&\sum_{k=0}^{2}|z|^k\E X^k+\frac{|z|}{2\pi}\sum_{k=2}^{\infty}|z|^{k}\int_{r_{-}}^{1}x^{\lfloor k\rfloor_e}(1-x)^{-1/2}dx\label{trickyintegral}\\
&=&\sum_{k=0}^{2}|z|^k\E X^k+\frac{|z|}{2\pi}\sum_{k=2}^{\infty}|z|^{k}2(1-r_{-})^{1/2}\sum_{j=0}^{\lfloor k\rfloor_e}\frac{(-1)^{j}\binom{\lfloor k\rfloor_e}{j}(1-r_{-})^{\lfloor k\rfloor_e-j}}{2(\lfloor k\rfloor_e-j)+3}.\nonumber
\end{eqnarray}
The integral in line~(\ref{trickyintegral}) is given by equation 2.221 in \cite{GR00}. 
One may verify that 
\begin{equation*}
\frac{\binom{n}{k}(1-r_+)^{n-k}}{2(n-k)+3}-\frac{\binom{n}{k-1}(1-r_+)^{n-k+1}}{2(n-k+1)+3}< 
\binom{n}{k}(1-r_+)^{n-k}-\binom{n}{k-1}(1-r_+)^{n-k+1},\nonumber
\end{equation*}
so that 
\begin{equation*}
\sum_{j=0}^{\lfloor k\rfloor_e}\frac{(-1)^{j}\binom{\lfloor k\rfloor_e}{j}(1-r_{-})^{\lfloor k\rfloor_e-j}}{2(\lfloor k\rfloor_e-j)+3}
\leq \sum_{j=0}^{\lfloor k\rfloor_e}(-1)^{j}\binom{\lfloor k\rfloor_e}{j}(1-r_{-})^{\lfloor k\rfloor_e-j}=r_{-}^{\lfloor k\rfloor_e}.
\end{equation*}
Thus
\begin{equation*}
(\ref{trickyintegral})\leq\sum_{k=0}^{2}|z|^k\E X^k+\frac{|z|}{2\pi}(1-r_{-})^{1/2}\sum_{k=2}^{\infty}|z|^{-k}r_{-}^{\lfloor k\rfloor_e}\nonumber,
\end{equation*}
which converges for all $|z|<1/r_{-}$. 
This gives 
\begin{equation}
\mu(1)=(1+z)\etapq(z)-(1+z)\sum_{k=0}^{\infty}z^k\int_{r_{-}}^{r_{+}}x^{k}\EXP dx\label{moments},
\end{equation}
where equation~(\ref{moments}) holds for all values $|z|<1$, and the sum on the right side of 
equation~(\ref{moments}) remains finite as $z\rightarrow -1$.
Convergence of the necessary sums and integrals is now established for all $p,q\in(0,1)$.
Using the equation for $\etapq(z)$ from Proposition~\ref{isequiv}, 
\begin{equation*}
(1+z)\etapq(z)=\frac{1+(p+q)z+\sqrt{1+(2(p+q)-4pq)z+((p+q)^2-4pq)z^{2}}}{2}.
\end{equation*}
Allowing $z\rightarrow - 1$, we have 
\begin{eqnarray}
\mu(1)&=&\frac{1-(p+q)+\sqrt{(1-(p+q))^2}}{2}\nonumber\\
&=&\frac{1-(p+q)+|1-(p+q)|}{2} \label{leftover}.
\end{eqnarray}
When $p+q\geq 1$, equation~(\ref{leftover}) is equal to $0$, 
and when $p+q<0$, equation~(\ref{leftover}) is equal to $1-\min(p,q)$.
From equations~(\ref{max}) and~(\ref{leftover}), it follows that 
when $p+q> 1$
\begin{equation*}
\int_{r_{-}}^{r_{+}}\EXP dx =1-\max(p,q),
\end{equation*}
and when $p+q\leq 1$,
\begin{equation*}
\int_{r_{-}}^{r_{+}}\EXP dx =\min(p,q).
\end{equation*}
Now it only remains to remove the point mass at $0$ and normalize the distribution by $1/(1-\max(p,q))$.
\end{proof}

For the following corollary we define the $n$ singular values of the matrix 
$\Pn\Fn\Qn$ to be the (positive) 
square roots of the $n$ eigenvalues of the matrix $\Pn\Fn\Qn\Fnstar\Pn$. 
We thus have the following limiting distribution for the singular values of $\Pn\Fn\Qn$.
\begin{corollary}
For $i=1,...,n$ let $i$ be contained in $\On$ independently with probability $(1-q)$ and, also independently, 
let $i$ be included in $\Tn$ with probability $(1-p)$. 
Then the empirical distributions of the $\min(|\Tn|,|\On|)$ largest singular values of $\FOnTn $ and $\UbOnTn $ converge almost surely to 
\begin{small}\begin{equation*}
f^s_{p,q}(x)=\frac{\sqrt{x^2(1-\frac{r^s_{-}}{x})(\frac{r^s_{+}}{x}-1)}}{2\pi(1-x^2)(1-\max(p,q))}\cdot I_{(r^s_{-},r^s_{+})}(x)+\frac{\max(0,1-(p+q))}{1-\max(p,q)}\cdot\delta(x-1)
\end{equation*}\end{small}
where
\begin{equation*}
r^s_{-}=|\sqrt{p(1-q)}-\sqrt{q(1-p)}|
\end{equation*}
and
\begin{equation*}
r^s_{+}=\sqrt{p(1-q)}+\sqrt{q(1-p)}.
\end{equation*}
\end{corollary}

\begin{proof}
The measures for the singular values $0$ and $1$ are the same as the eigenvalues 
$0$ and $1$ of the previous theorem. 
Also, the continuous part of the measure will clearly have support 
$[r^s_{-},r^s_{+}]=[\sqrt{r_{-}},\sqrt{r_{+}}]$.
Using equation~(\ref{factors}), we have 
\begin{eqnarray*}
f^{s}_{p,q}(x)&=&2 xf_{p,q}(x^2)\\
&=&\Im \; 2x \frac{\sqrt{(1-\frac{r_{-}}{x^{2}})(1-\frac{r_{+}}{x^{2}}  )}}{2\pi(1-x^{2})}\nonumber\\
&=&\Im\;x\frac{\sqrt{(1-\frac{r_{-}}{x^2})(\frac{r_{+}}{x^2} -1 )}}{\pi(1-x^2)}\nonumber\\
&=& \frac{\sqrt{x^2(1-\frac{r_{-}}{x^{2}})(\frac{r_{+}}{x^{2}}-1)}}{\pi(1-x^{2})}\cdot I_{[\sqrt{r_{-}},\sqrt{r_{+}}]}(x).
\end{eqnarray*} 
\end{proof}

\end{section}

\noindent\textit{Acknowledgement} 
The author thanks Roland Speicher, Walid Hachem, 
Antonia Tulino and Sergio Verd\'u  for helpful discussions and comments.


\begin{thebibliography}{10}

\bibitem{AGZ09}
G.~Anderson, A.~Guionnet, and O.~Zeitouni, \emph{An introduction to random
  matrices}, Cambridge University Press, Cambridge, 2009.

\bibitem{BY88}
Z.~D. Bai and Y.~Q. Yin, \emph{Necessary and sufficient conditions for almost
  sure convergence of the largest eigenvalue of a {W}igner matrix}, Ann.
  Probab. \textbf{16} (1988), no.~4, 1729--1741.

\bibitem{BY93}
Z.~D. Bai and Y.~Q. Yin, \emph{Limit of the smallest eigenvalue of a large-dimensional sample
  covariance matrix}, Ann. Probab. \textbf{21} (1993), no.~3, 1275--1294.

\bibitem{Can06ICM}
Emmanuel~J. Cand{\`e}s, \emph{Compressive sampling}, International {C}ongress
  of {M}athematicians. {V}ol. {III}, Eur. Math. Soc., Z\"urich, 2006,
  pp.~1433--1452.

\bibitem{GG84}
A.Yu. Garnaev and E.D. Gluskin, \emph{{On widths of the Euclidean ball.}}, Sov.
  Math., Dokl. \textbf{30} (1984), 200--204 (English. Russian original).

\bibitem{Gem80}
Stuart Geman, \emph{A limit theorem for the norm of random matrices}, Ann.
  Probab. \textbf{8} (1980), no.~2, 252--261.

\bibitem{Glu84}
E.D. Gluskin, \emph{{Norms of random matrices and widths of finite-dimensional
  sets.}}, Math. USSR, Sb. \textbf{48} (1984), 173--182 (English).

\bibitem{GR00}
I.S. Gradshteyn and I.M. Ryzhik, \emph{Table of integrals, series, and
  products}, Academic Press, New York, 2000.

\bibitem{GS77}
Ulf Grenander and Jack~W. Silverstein, \emph{Spectral analysis of networks with
  random topologies}, SIAM J. Appl. Math. \textbf{32} (1977), no.~2, 499--519.

\bibitem{MP67}
V.~A. Mar{\v{c}}enko and L.~A. Pastur, \emph{Distribution of eigenvalues in
  certain sets of random matrices}, Mat. Sb. (N.S.) \textbf{72 (114)} (1967),
  507--536.

\bibitem{PR05}
D\'enes Petz and J\'ulia R\'effy, \emph{{Large deviation for the empirical
  eigenvalue density of truncated Haar unitary matrices.}}, Probab. Theory
  Relat. Fields \textbf{133} (2005), no.~2, 175--189 (English).

\bibitem{RV08c}
M.~Rudelson and R.~Vershynin, \emph{The {L}ittlewood-{O}fford problem and
  invertibility of random matrices}, Adv. Math. \textbf{218} (2008), no.~2,
  600--633.

\bibitem{Tao05}
T.~Tao, \emph{An uncertainty principle for cyclic groups of prime order}, Math.
  Res. Lett. \textbf{12} (2005), no.~1, 121--127.

\bibitem{Tro08}
Joel~A Tropp, \emph{On the linear independence of spikes and sines}, Journal of
  Fourier Analysis and Applications \textbf{14} (2008), no.~5, 838--858.

\bibitem{TCSV10}
A.~Tulino, G.~Caire, S.~Shamai, and S.~Verd\'u, \emph{Capacity of channels with
  frequency-selective and time-selective fading}, IEEE Trans. Information
  Theory (2010), to appear.

\bibitem{TV04}
A.~Tulino and S.~Verd\'u, \emph{Random matrices and wireless communications},
  Foundations and Trends in Information Theory \textbf{1} (June 2004), no.~1.

\bibitem{TVCS07}
A.~Tulino, S.~Verd\'u, G.~Caire, and S.~Shamai, \emph{Capacity of the gaussian
  erasure channel}, Preprint, Technion CCIT no. 655 (2007),
  http://www3.ee.technion.ac.il/ccit/info/Publication/Scientific\_e.asps.

\bibitem{Wig55}
Eugene~P. Wigner, \emph{Characteristic vectors of bordered matrices with
  infinite dimensions}, Ann. of Math. (2) \textbf{62} (1955), 548--564.

\bibitem{Wig58}
Eugene~P. Wigner, \emph{On the distribution of the roots of certain symmetric matrices},
  Ann. of Math. (2) \textbf{67} (1958), 325--327.

\bibitem{Wis28}
J.~Wishart, \emph{The generalized product moment distribution in samples from a
  normal multivariate population}, Biometrika \textbf{20A} (1928), 32--52.

\bibitem{YBK88}
Y.~Q. Yin, Z.~D. Bai, and P.~R. Krishnaiah, \emph{On the limit of the largest
  eigenvalue of the large-dimensional sample covariance matrix}, Probab. Theory
  Related Fields \textbf{78} (1988), no.~4, 509--521.

\bibitem{ZS00}
Karol {\.Z}yczkowski and Hans-J{\"u}rgen Sommers, \emph{Truncations of random
  unitary matrices}, J. Phys. A \textbf{33} (2000), no.~10, 2045--2057.

\end{thebibliography}

\end{document}